\documentclass[reqno]{amsart}

\usepackage[utf8]{inputenc}

\usepackage{amssymb , amsmath, amsthm, esint}

\usepackage{latexsym}

\newtheorem{defi}{Definition}[section] 
\newtheorem{thm}[defi]{Theorem}

\newtheorem{rem}[defi]{Remark}
 \newtheorem{prop}[defi]{Proposition}
 \newtheorem{lemma}[defi]{Lemma}

\newcommand{\R}{\mathbb R}
\newcommand{\N}{\mathbb N}

\title[Compact manifolds with large Steklov eigenvalues]{Compact manifolds with fixed boundary and large Steklov eigenvalues}
\author{Bruno Colbois}
\address{Universit\'e de Neuch\^atel, Institut de Math\'ematiques, Rue
  Emile-Argand 11, CH-2000 Neuch\^atel, Switzerland}
\email{bruno.colbois@unine.ch}
\author{Ahmad El Soufi}
\thanks{During the first week of 2017, AG and BC were supposed to
travel to Tours and work with Ahmad El Soufi to complete this paper. We learned just a few
days before our visit of his untimely death. Ahmad was a colleague and a
friend. He will be dearly missed.}
\address{Laboratoire de Math\'{e}matiques et Physique Th\'{e}orique,
  UMR-CNRS 6083, Universit\'{e} François Rabelais de Tours, Parc de
  Grandmont, 37200 Tours, France} 
\email{elsoufi@univ-tours.fr}
\author{Alexandre Girouard}
\address{Département de mathématiques et de statistique, 
Pavillon Alexandre-Vachon, Université Laval,
Québec, QC, G1V 0A6, Canada}
\email{Alexandre.Girouard@ulaval.ca}

\begin{document}
\begin{abstract}
  Let $(M,g)$ be a compact Riemannian manifold with boundary. Let
  $b>0$ be the number of connected components of its boundary. For
  manifolds of dimension $\geq 3$, we prove that for $j=b+1$ it is
  possible to obtain an arbitrarily large Steklov eigenvalue
  $\sigma_j(M,e^\delta g)$ using a conformal perturbation $\delta\in
  C^\infty(M)$ which is supported in a thin neighbourhood of the
  boundary, with $\delta=0$ on the boundary. For $j\leq b$, it is also
  possible to obtain arbitrarily large eigenvalues, but the conformal
  factor must spread throughout the interior of $M$.
  In fact, when working in a fixed conformal class and for $\delta=0$
  on the boundary, it is known
  that the volume of $(M,e^\delta g)$ has to tend to infinity in order
  for some $\sigma_j$ to become arbitrarily large. 
  This is in stark
  contrast with the situation for the eigenvalues of the Laplace
  operator on a closed manifold, where a conformal factor that is large enough for the volume
  to become unbounded results in the spectrum collapsing to 0.
  We also prove that
  it is possible to obtain large Steklov eigenvalues while keeping different
  boundary components arbitrarily close to each other, by
  constructing a convenient Riemannian submersion.
\end{abstract}
\maketitle


\section{Introduction}

The Steklov eigenvalues of a smooth compact connected Riemannian
manifold $(M,g)$ of dimension $n+1\geq 2$ with boundary $\Sigma$ are the real
numbers $\sigma$ for which there exists a non-zero harmonic function
$f:M\rightarrow\R$ which satisfies $\partial_\nu f=\sigma f$ on the
boundary $\Sigma$. Here and further $\partial_\nu$ is the outward
normal derivative on $\Sigma$. It is well known that the Steklov
eigenvalues form a discrete spectrum $0=\sigma_1<\sigma_2\leq
\sigma_3\leq\cdots\nearrow\infty$, where each eigenvalue is repeated
according to its multiplicity. The interplay between the geometry of
$M$ and the Steklov spectrum has recently attracted substantial
attention. See~\cite{gpsurvey} and the references therein for recent
development and open problems.

Many developments linking the Steklov eigenvalues of a
compact manifold $M$ with the eigenvalues $\lambda_j$ of the Laplace
operator on its boundary have appeared. See for instance
\cite{wangxia2,ceg2} and more recently \cite{prostu}, where it was
proved that for any Euclidean domain $\Omega\subset\R^{n+1}$ with
smooth boundary, there exists a constant $c_\Omega>0$ such that
$$\lambda_j\leq \sigma_j^2+2c_\Omega\sigma_j,\qquad\mbox{and}\qquad
\sigma_j\leq c_\Omega+\sqrt{c_\Omega^2+\lambda_j}.$$
These results indicate a strong link between the Steklov eigenvalues
of a manifold and the geometry of its boundary. {See also \cite{cgh,Xi} for recent similar results on Riemannian manifolds.}
In fact, on smooth surfaces the spectral asymptotics is
completely determined by the geometry of the boundary \cite{GPPS}.


\medskip
In the present paper, we investigate the following question:
\smallskip

\emph{For a given closed Riemannian manifold $\Sigma$, how large can $\sigma_2(M)$ be among compact Riemannian manifolds $M$ with boundary isometric to $\Sigma$?}

\smallskip
For a manifold $(M,g)$ of dimension $n+1\geq 3$ with boundary
$\Sigma$, we will see that it is possible to make $\sigma_2$
arbitrarily large by using conformal perturbations $g'=h^2g$ such that
$h=1$ on $\Sigma$. Of course, this imply that any eigenvalue
$\sigma_j$ becomes arbitrarily large under such a conformal
perturbation, but the situation is more interesting than that.
Indeed, let $b>0$ be the number of connected components. We will prove
that it is possible to make $\sigma_{b+1}$ arbitrarily large by using a conformal
perturbation $h^2g$ where $h$ is a smooth function which is different
from 1 only in a thin strip located arbitrarily close to the
boundary (with $h=1$ identically on $\Sigma$). It is also possible to
make lower eigenvalues $\sigma_j$ arbitrarily large, but this requires conformal
perturbations which penetrates deeply into the manifold $M$. 
 
One could also ask how small an eigenvalue $\sigma_j(M)$ can be. This
question is easier, as it is relatively easy to construct small
eigenvalues while keeping the boundary fixed. On surfaces, it is
sufficient to create thin passages (see Figure 3 of \cite[Section
4]{gpsurvey}) while for manifolds of dimension $\geq 3$, one can use a
conformal perturbation supported inside the manifold $M$. See
Proposition \ref{proposition:smalleigenvalues}. 

\subsubsection*{Large eigenvalues on surfaces} 
It was proved in \cite{kokarev} that any compact surface $M$ with boundary of length $L>0$ satisfies
\begin{gather}\label{ineq:kokarev}
\sigma_2(M)\leq\frac{8\pi}{L}(1+\mbox{genus}(M)).
\end{gather}
In \cite{cgr}, a sequence of surfaces $(M_l)_{l\in\N}$ with one boundary component of fixed length $L>0$ was constructed, which satisfies
$$\lim_{l\rightarrow\infty}\sigma_2(M_l)=+\infty.$$
These two results give a complete answer to our initial question for surfaces: it is possible to obtain arbitrarily large $\sigma_2$, but it is necessary to increase the genus of $M$ in order to do so.

\subsubsection*{Manifolds of higher dimensions}
For any compact Riemannian manifold $(M,g)$ of dimension $\geq 3$ with boundary $\Sigma$, we will show that a conformal perturbation is sufficient to obtain arbitrarily large $\sigma_2$. More can be said: let $b>0$ be the number of connected components of the boundary $\Sigma$. The next theorem shows that it is possible to make $\sigma_{b+1}$ arbitrarily large using conformal perturbations $g_\varepsilon$ which are supported in an arbitrary neighbourhood of the boundary $\Sigma$, and which coincide with $g$ on the boundary. It is also possible to make $\sigma_2$ large, but this requires conformal perturbations away from the boundary (See Proposition \ref{prop:necesbsmall}). The next theorem is the main result of this paper.

\begin{thm}\label{thm:main}
Let $(M,g)$ be a compact connected Riemannian manifold of dimension $\geq 3$ with  $b\in\N$ boundary components 
$\Sigma_1,\cdots,\Sigma_b$. 

\smallskip
\noindent (i) For every neighborhood $V$ of $\Sigma$, there exists a one-parameter family of 
 Riemannian metrics $g_\varepsilon $  conformal to $g$  which coincide with $g$ on $\Sigma$ and in $M\setminus V $, such that 
 $$\sigma_{b+1}(g_\varepsilon) \to\infty \quad as\quad \varepsilon\to 0.$$
 
\smallskip
\noindent (ii) There exists a one-parameter family of 
 Riemannian metrics $g_\varepsilon$  conformal to $g$  which coincide with $g$ on $\Sigma$ such that 
 $$\sigma_{2}(g_\varepsilon) \to\infty \quad as\quad \varepsilon\to 0.$$
\end{thm}
The proof of Theorem \ref{thm:main} will be presented in Section \ref{section:largesteklovfixedbdry}.
It is important to note that in order to obtain large eigenvalues, it is necessary to perturb the metric near each points of the boundary.
\begin{prop}\label{proposition:bdrynecessary}
Let $(M,g)$ be a compact Riemannian manifold with boundary $\Sigma$. Let $p\in\Sigma$ and let $\varepsilon>0$. Then any Riemannian metric $g'$ on $M$ which {coincides} with $g$ on $B(p,\varepsilon)\subset M$ {satisfies
$$\sigma_k(M,g')\leq \lambda_k^D(B(p,\varepsilon),g),$$
where $\lambda_k^D(B(p,\varepsilon),g)$ is the $k$-th eigenvalue of a
mixed Steklov-Dirichlet eigenvalue problem.}
\end{prop}
The proof of this observation is an exercise in the use of min-max characterisations of eigenvalues. It will be presented in Section \ref{section:variationalchar}.

\medskip
The conformal perturbations which are used in the proof of Theorem
\ref{thm:main} are such that the volume $|M|_{g_\varepsilon}$ tends to
infinity as $\varepsilon\to 0$. This is a necessary condition when
working in a fixed conformal class $[g]$. Indeed, the following
inequality for $g'\in[g]$ was proved in \cite{asma}: 
$$\sigma_k(M,g')|\Sigma|^{1/n}<\frac{A+Bk^{\frac{2}{n+1}}}{I(M)^{(n-1)/n}},$$
where $A$ is a constant which depends on $g$, $B$ depends on the dimension and $I(M)$ is the isoperimetric ratio
$$I(M)=\frac{|\Sigma|_{g'}}{|M|_{g'}^{n/(n+1)}}.$$
In each of the constructions, the diameter also becomes unbounded.
In Theorem~\ref{thm:warpedprod}, we
construct a sequence $g_m$ of Riemannian metrics on $M$ such that
$g_m\bigl|\bigr._\Sigma=g_\Sigma$ and such that  $(M,g_m)$ has
uniformly bounded diameter and $\sigma_2$ becomes arbitrarily large.  

To conclude this introduction, note that it is difficult to obtain lower bounds for Steklov eigenvalues. Under relatively strong convexity assumptions this was already investigated by Escobar in \cite{esco1}. More recently Jammes \cite{jammes1} proposed an interesting inequality in the spirit of Cheeger:
$$\sigma_2(M)\geq\frac{h(M)j(M)}{4}.$$
Here $h(M)$ is the classical Cheeger constant and $j(M)$ was introduced in \cite{jammes1}.
It is challenging to obtain effective lower bounds on $\sigma_2(M)$ using this inequality, mainly because it is difficult to estimate $h(M)$ and $j(M)$. Moreover, it it is interesting to note that the metrics which we construct in Theorem \ref{thm:main} and \ref{thm:warpedprod} have small Cheeger and Cheeger-Jammes constants, despite the eigenvalue $\sigma_2$ being arbitrarily large.

\subsection{Plan of the paper}
In the next section, we present the variational characterization of
the Steklov and mixed Steklov-Dirichlet eigenvalue problems and deduce
some simple consequences. In Section
\ref{section:largesteklovfixedbdry} we prove the main results of the
paper by first working in cylinders and then using quasi-isometric
control of eigenvalues to obtain Theorem \ref{thm:main}. We also prove
Theorem \ref{thm:warpedprod} which provides an example where two
boundary components are arbitrarily close to each other.

\section{Variational characterisation and quasi-isometric control of eigenvalues}
\label{section:variationalchar}
Let $M$ be a smooth compact Riemannian manifold with boundary
$\Sigma$. Let $\mathcal{H}_k(M)$ be the set of $k$-dimensional linear
subspaces of $C^\infty(M)$. 
It is well known that the Steklov eigenvalue $\sigma_k$ is given by
\begin{gather}\label{varchar:steklov}
 \sigma_k(M,g)=\min_{E\in\mathcal{H}_k}\max_{0\neq f\in E}\frac{\int_M|df|_g^2\,dv_g}{\int_\Sigma|f|^2\,dv_g},
\end{gather}
where $dv_g$ is the volume form. It the following we will use conformal metrics of the form $g'=h^2g$, where $h$ is a smooth function on $M$ such that $h=1$ identically on the boundary $\Sigma$. In the min-max characterization of $\sigma_k(M,h^2g)$, the denominator is the same as above, while the numerator is
$$\int_M|df|_{g'}^2\,dv_{g'}=\int_M|df|_g^2h^{n-1}\,dv_g.$$
In the following, we will often write $|df|$ for $|df|_g$.
 
 \medskip
We seize this opportunity to prove one of the simple statement from the introduction.
\begin{prop}\label{proposition:smalleigenvalues}
  Let $M$ be a compact smooth Riemannian manifold of dimension $n+1$ at least
  3, with boundary $\Sigma$.  For each $p\in\Sigma$ and each
  $\varepsilon>0$, there exists a sequence of conformal deformations
  $h_m^2g$ such that $h_m>0$ is a smooth function which is identically equal to 1
  on $\Sigma$ and on the complement of the ball $B(p,\varepsilon)\subset
  M$, and such that $\lim_{m\rightarrow\infty}\sigma_k(M,h_m^2g)=0$ for
  each $k\in\N$.  
\end{prop}
\begin{proof}
Given $\varepsilon>0$, let $p\in\Sigma$ and consider a smooth function
$f\in C^{\infty}(M)$ which is supported in $B(p,\varepsilon)\subset M$
and which does not vanish at $p$. Let $h_m$ be a sequence of positive
smooth functions on $M$ such that $h_m=1$ on $\Sigma$ {and on the complement of $B(p,\varepsilon)$, which satisfies}
$\lim_{m\rightarrow\infty}h_m=0$ uniformly on compact subsets of
$B(p,\varepsilon)\cap\mbox{interior}(M)$. It follows that the conformal
deformations $\tilde{g}_m=h_m^2g$ satisfy 
\begin{gather*}
\lim_{m\rightarrow\infty}\frac{\int_M|df|_{\tilde{g}}^2\,dv_{\tilde{g}_m}}{\int_{\Sigma} f^2\,dv_{\tilde{g}_m}}
=\frac{\int_M|df|_g^2\,h_m^{n-1}dv_{g}}{\int_{\Sigma} f^2\,dv_{g}}=0.
\end{gather*}
Using $k$ functions $f_j\quad (j=1,\cdots,k)$ with disjoint support in $B(p,\varepsilon)$ instead of a single function $f$, the result now follows from the min-max characterization of $\sigma_k$.
\end{proof}

We will also use the following mixed Steklov-Dirichlet problem on a domain $\Omega\subset M$:
$$\Delta f=0\mbox{ in }\Omega,\quad f=0 \mbox{ on }\partial\Omega\setminus\Sigma,\quad
\partial_\nu f=\lambda f\mbox{ on }\partial\Omega\cap\Sigma.$$
It has discrete spectrum
$0<\lambda_1^D\leq\lambda_2^D\leq\cdots\nearrow\infty$.
The $k$-th eigenvalue is given by 
\begin{gather}
 \lambda_k^D=\min_{E\in\mathcal{H}_{k,0}}\max_{0\neq f\in E}\frac{\int_{\Omega}|\nabla f|^2\,dv_g}{\int_{\partial\Omega\cap\Sigma}f^2\,dv_g},
\end{gather}
where $\mathcal{H}_{k,0}=\{E\in\mathcal{H}_k\,:\,f=0\mbox{ on }\partial\Omega\setminus\Sigma\times\{0\}\quad \forall f\in E\}$. 
For more {information} on mixed Steklov problems see for instance \cite{bankulcpolt} and \cite{agranovich}.

\medskip
We are now ready for the proof of Proposition \ref{proposition:bdrynecessary}.
\begin{proof}[Proof of Proposition \ref{proposition:bdrynecessary}]
Let $(\phi_k)$ be a sequence of eigenfunctions corresponding to $\lambda_k^D(B(p,\varepsilon))$, which are extended by 0 elsewhere in $M$. Using the subspace $E_k=\mbox{span}(\phi_1,\cdots,\phi_k)$ in the min-max characterization of $\sigma_k(M,g')$ completes the proof.
\end{proof}

The following proposition is borrowed from \cite{cgr}. It is classical and follows directly from the min-max characterization of the eigenvalues. We believe that this principle was used for the first time in \cite{dodziuk}, in the context of the Laplace-Beltrami operator acting on differential forms.
\begin{prop}\label{proposition:quasiisocontrol}
 Let $M$ be a compact manifold of dimension $n$, with smooth boundary $\Sigma$ and let $g_1,g_2$ be two Riemannian metrics on $M$ which are quasi-isometric with ratio $A\geq 1$, which means that for each $x\in M$ and $0\neq v \in T_xM$ we have
 $$
 \frac{1}{A} \le \frac{g_1(x)(v,v)}{g_2(x)(v,v)} \le A.
 $$
 Then the Steklov eigenvalues with respect to $g_1$ and $g_2$ {satisfy} the following inequality:
 $$
 \frac{1}{A^{2n+1}}\le \frac{\sigma_k(M,g_1)}{\sigma_k(M,g_2)} \le A^{2n+1}.
 $$
\end{prop}
Note also that if the metrics $g_1$ and $g_2$ are quasi-isometric with ratio $A\geq 1$, then given a smooth function $h$ on $M$, the two conformal metrics $h^2g_1$ and $h^2g_2$ are also quasi-isometric with the same ratio $A$. This will be useful in the proof of Theorem \ref{thm:main} when going from cylindrical boundaries to arbitrary manifolds.

\section{Large Steklov eigenvalues on manifolds with fixed boundary}
\label{section:largesteklovfixedbdry}

Let $M$ be a compact manifold of dimension $(n+1)$ with $b\geq 1$
boundary components: $\Sigma= \Sigma_1 \cup ... \cup \Sigma_b$. We will prove Theorem \ref{thm:main} by first working under the extra hypothesis that the boundary $\Sigma$ of $M$ has a neighbourhood which is isometric to the product $\Sigma\times [0,L)$ for
some  $L>0$. This is not a strong hypothesis since it is always
satisfied up to a quasi-isometry (See the proof of Theorem \ref{thm:main} below).

\medskip
In the present context, we {denote} $g_0$ the restriction of the Riemannian metric $g$ to the boundary $\Sigma$, and $d_0$ the corresponding exterior derivative on $C^\infty(\Sigma)$. The spectrum of the Laplace operator on $\Sigma$ is denoted 
$$0=\lambda_1=\cdots=\lambda_b<\lambda_{b+1}\leq \cdots\rightarrow+\infty.$$
Theorem \ref{thm:main} will follow from Proposition \ref{Conf1} and Proposition \ref{Conf2} below.
\begin{prop}\label{Conf1}
  Let $(M,g)$ be a Riemannian manifold of dimension $n+1\ge 3$, with boundary
  $\Sigma$ and assume that there exists a neighborhood $V$  of
  $\Sigma$ which is isometric to the product $\Sigma\times [0,L)$ for
  some  $L>0$.   
 For every  $\varepsilon>0$ sufficiently small, there
  exists a Riemannian metric $g_\varepsilon=h_\varepsilon^2g$  conformal
  to $g$  which coincides with $g$ in the complement of
  $\Sigma\times (\varepsilon, 4\varepsilon) $ 
  and such that
  $$\sigma_{b+1}(g_\varepsilon) \ge \frac A \varepsilon,$$
  where $A=\frac 1{4}\min\{ \lambda_{b+1}(\Sigma), \frac1{4}\}>0$.
\end{prop}
\begin{proof}
For every positive $\varepsilon < \min\{\frac{L}{4},\frac{2}{L}\}$,
  define a Riemannian metric $g_\varepsilon=h^2_\varepsilon g$ on $M$
  where $h_\varepsilon\geq 1$ is a smooth function which is  identically equal
  to $1$ in the complement of $\Sigma\times (\varepsilon, 4\varepsilon)$ and, for $(x,t)\in
  \Sigma\times [2\varepsilon,3\varepsilon]$, 
    \begin{gather*}
      h_\varepsilon(x,t)=\varepsilon^{-2}.
    \end{gather*}
    Let $\{\phi_k\}_{k\in \N}$ be an orthonormal basis of eigenfunctions
    of the Laplacian on $\Sigma$, with
    $\Delta\phi_k=\lambda_k\phi_k$. Denote by $\Sigma_1, \dots,
    \Sigma_b$ the connected components of $\Sigma$.  One has
    $\lambda_1=\cdots=\lambda_b=0$ and, for every $j\le b$,  one
    chooses $\phi_j=\vert\Sigma_j\vert^{-\frac 12}$ on $\Sigma_j$ and
    $\phi_j=0$ elsewhere.

Let $f$ be a smooth function on $M$ with $\int_\Sigma f dv_{g_0}=0$ and
$\int_\Sigma f^2 dv_{g_0}=1$.  The restriction of
$f$ to $\Sigma\times [0,L)$  is developed in Fourier series:  
\begin{equation*}
f(x,t)=\sum_{j\ge 1} a_j(t)\phi_j(x)
\end{equation*} 
with
\begin{equation*}
a_j(0)=\vert\Sigma_j\vert^{-\frac 12}\int_{\Sigma_j}f dv_{g_0}\ , \quad\mbox{ for } j=1,\dots, b
\end{equation*} 
and, since $\int_\Sigma f dv_{g_0}=0$ and $\int_\Sigma f^2 dv_{g_0}=1$ 
\begin{equation}\label{conf2} 
  a_1(0)\vert\Sigma_1\vert^{\frac 12}+\cdots a_b(0)\vert\Sigma_b\vert^{\frac 12}=0
\end{equation} 
\begin{equation}\label{conf2'} 
  \sum_{j\ge 1}a_j(0)^2=1.
 \end{equation}
Observe that   $\sum_{j=1}^b a_j(0)^2$ is the square of the $L^2$-norm of
the orthogonal projection of $f\bigl|\bigr._\Sigma$ on
$\ker(\Delta)=\mbox{span}\{\phi_1,\dots,\phi_b\}$, in $L^2(\Sigma,g_0)$.  
From
$$df(x,t)=\sum_{j\ge 1} \left(a'_j(t)\phi_j(x)\,dt +
a_j(t)\,d_0\phi_j(x)\right)$$
and $\int_\Sigma\vert d_0\phi_j\vert^2 dv_{g_0}=\lambda_j$, it follows that the Dirichlet energy of $f$ on $(M,h_\varepsilon^2g)$ is 
\begin{align}\label{conf3} 
  R_\varepsilon(f):=
  \int_M \vert df\vert^2h^{n-1}_\varepsilon dv_{g}
  &\ge\int_{\Sigma\times (0,L)} \vert df\vert^2h^{n-1}_\varepsilon dv_{g}\nonumber\\
  &{=}\sum_{j\ge 1}  \int_0^L \left(a'_j(t)^2+ \lambda_j a_j(t)^2\right)h^{n-1}_\varepsilon(t) dt.
\end{align} 
At this point, observe that either the function $a_j$ decreases quickly when moving away from the boundary (which costs energy from the first term in \eqref{conf3}) or it remains big enough, and the second term contributes a large amount to the energy $R_\varepsilon(f)$. This is now explained more precisely.

\medskip
Fix an integer $j\geq 1$. If $\vert a_j(t)\vert \ge \frac 12\vert
a_j(0)\vert$ for all $t\in [2\varepsilon, 3\varepsilon)$, then  
\begin{align}\label{conf4} 
  \int_0^L \lambda_j a_j(t)^2 h^{n-1}_\varepsilon(t) dt&\ge
  \lambda_j\int_{2\varepsilon}^{3\varepsilon} a_j(t)^2\varepsilon^{-2(n-1)}dt\nonumber\\
  &\ge  \frac {\lambda_j}4\varepsilon^{-2n+3} a_j(0)^2\ge \frac {\lambda_j}{4\varepsilon } a_j(0)^2.
\end{align} 
Otherwise, there exists $t_0\in [2\varepsilon, 3\varepsilon)$ with
  $\vert a_j(t_0)\vert \le \frac 12\vert a_j(0)\vert$, which implies
  $\vert  a_j(0) -a_j(t_0)\vert\ge \vert a_j(0)\vert- \vert
  a_j(t_0)\vert\ge \frac12\vert a_j(0)\vert $ and using the Cauchy-Schwarz inequality this leads to
\begin{align}\label{conf5} 
  \int_0^L a'_j(t)^2h^{n-1}_\varepsilon (t)dt
  &\ge\int_0^{2\varepsilon} a'_j(t)^2 dt + \int_{2\varepsilon}^{t_0} a'_j(t)^2 \varepsilon^{-2(n-1)}dt\\
  &\ge\frac 1{2\varepsilon} \left(\int_0^{2\varepsilon} a'_j(t)
  dt\right)^2+\frac {\varepsilon^{-2(n-1)}}{t_0-2\varepsilon}
  \left(\int_{2\varepsilon}^{t_0} a'_j(t) dt\right)^2\nonumber
\end{align}
where $\frac {\varepsilon^{-2(n-1)}}{t_0-2\varepsilon}\ge \frac {\varepsilon^{-2(n-1)}}{\varepsilon}{\geq}\frac1\varepsilon\ge \frac 1{2\varepsilon}$. Hence (since $x^2+y^2\ge\frac12(x+y)^2$)
\begin{align}\label{conf5'} 
  \int_0^L a'_j(t)^2h^{n-1}_\varepsilon(t)\,dt
  &\ge
  \frac 1{2\varepsilon} \left[\left(\int_0^{2\varepsilon} a'_j(t)\,dt\right)^2+ \left(\int_{2\varepsilon}^{t_0} a'_j(t)\, dt\right)^2\right]\nonumber\\
  &\ge \frac 1{4\varepsilon} \left(\int_0^{2\varepsilon} a'_j(t)\, dt+
  \int_{2\varepsilon}^{t_0} a'_j(t)\, dt\right)^2\nonumber\\
  &=\frac
  1{4\varepsilon}(a_j(0) -a_j(t_0))^2\ge \frac
  1{16\varepsilon}a_j(0)^2.
\end{align}
For $j\ge b+1$ one has $\lambda_j \ge \lambda_{b+1}$ and, combining
\eqref{conf4} and \eqref{conf5'} leads to 
$$ \int_0^L \left(a'_j(t)^2+ \lambda_j
a_j(t)^2\right)h^{n-1}_\varepsilon(t) dt
\ge  \min\{ \frac {\lambda_{b+1}}{4}, \frac1{16}\}\frac 1\varepsilon a_j(0)^2
=\frac {A}\varepsilon a_j(0)^2.$$
Therefore, thanks to \eqref{conf3} and \eqref{conf2'}, 
\begin{equation}\label{conf5!}
R_\varepsilon(f)\ge\frac {A}\varepsilon \sum_{j\ge b+1}a_j(0)^2=\frac {A}\varepsilon \left(1-\sum_{j\le b}a_j(0)^2\right).
\end{equation}
Now,   every normalized function $f$ which is orthogonal in $L^2(\Sigma)$ to  $\ker (\Delta)=\mbox{span}\{\phi_1,\dots,\phi_b\}$, satisfies $\sum_{j\le b}a_j(0)^2=0$ and, then $R_\varepsilon(f)\ge \frac {A}\varepsilon$. Using the max-min principle we deduce that $\sigma_{b+1}(M,g_\varepsilon) \ge \frac {A}\varepsilon$. This completes the proof of (i).

\end{proof}

In Proposition \ref{Conf2} below, we will prove that it is also possible to make $\sigma_2$ arbitrarily large using a conformal perturbation $h^2g$. This is more difficult than for $\sigma_{b+1}$. One of the difficulties comes from the fact that the conformal perturbation will need to be supported everywhere inside the manifold $M$. This follows from the following easy proposition. 
\begin{prop}\label{prop:necesbsmall}
For every
  Riemannian metric $g'$   which coincides with $g$ on $\Sigma$ and
  in the complement of  $\Sigma\times [0,\frac{L}2)$, one has 
  $$ \sigma_{b}(g') \le \frac{2}{L}.$$
\end{prop}
\begin{proof}[Proof of Proposition \ref{prop:necesbsmall}]
For every $j\le b$, let $\psi_j$ be the function on $M$ such
that $\psi_j$ is  constant equal to zero in the complement of
$\Sigma_j\times [0,L)$ and,
for $(x,t)\in \Sigma_j\times  [0,L)$, 
\begin{equation}\nonumber
  \psi_j(x,t) =\left\{
  \begin{array}{lll}
    \vert\Sigma_j\vert^{-\frac 12} \qquad\qquad\quad \text{ in } \Sigma_j\times [0,\frac L 2],\\
    2(1-\frac tL) \vert\Sigma_j\vert^{-\frac 12}\quad\;  \text{ in }
    \Sigma\times [\frac L 2, L].
  \end{array}
  \right.
\end{equation} 
For any Riemannian metric $g'$ which coincides with $g$ on $\Sigma$
and on the complement of $\Sigma\times [0,\frac L 2]$, one has
$\int_M\vert d\psi_j\vert_{g'}^2 dv_{g'}= \frac 2L$ and $\int_\Sigma
\psi_j^2 dv_{g'}=1$. Moreover, $\psi_1,\cdots, \psi_b$ are mutually
orthogonal on the boundary. Therefore, using the min-max principle we
deduce that $\sigma_b(g')\le \frac 2L$. 
\end{proof}

\begin{prop}\label{Conf2}
  Let $(M,g)$ be a Riemannian manifold of dimension $n+1\geq 3$ with boundary
  $\Sigma$ and assume that there exists a neighbourhood of  $\Sigma$
  which is isometric to the product $\Sigma\times [0,L)$ for some
  $L>0$.  
 For every  $\varepsilon>0$ sufficiently small, there exists a
 Riemannian metric $g_\varepsilon=h_\varepsilon^2g$  conformal to $g$
 which coincides with $g$ in the neighbourhood
 $\Sigma\times [0,\varepsilon) $ of $\Sigma$ and
 such that $$\sigma_2(g_\varepsilon) \ge \frac C \varepsilon$$ 
 where $C$ is an explicit  constant which only depends on $g$.
\end{prop}

The following Poincaré type result will be useful.
\begin{lemma}\label{lemNeu}
Let $(M,g)$ be a compact manifold and denote by $\mu$ the first positive eigenvalue of the Laplacian of $(M,g) $ with Neumann boundary condition if $\partial M$ is nonempty. Let $V_1$ and $V_2$ be two disjoint measurable subsets of $M$ of positive volume. Every function  
 $f \in C^{\infty}(M)$ satisfies
 $$
\int_M \vert df\vert^2\, dv_g \ge \frac{\mu }{2}\min( {\vert V_1\vert_g}, {\vert V_2\vert_g}) \left(\fint_{V_1}f\, dv_g -\fint_{V_2}f\, dv_g \right)^2.
$$
where  $\fint_{M}f\, dv_g := \frac{1}{\vert M\vert_g}\int_{M}f\, dv_g$.
\end{lemma}

\begin{proof}[Proof of Lemma \ref{lemNeu}] 
Denote by $m=\fint_M f\, dv_g$ the mean value of $f$ on $M$. The function $f-m$ is orthogonal to constant functions on $M$ which implies 
$$\int_M \vert df\vert^2\, dv_g  \ge \mu \int_{M} (f-m)^2\, dv_g \ge \mu\int_{V_1} (f-m)^2\, dv_g +\mu\int_{V_2}(f-m)^2\, dv_g .$$
Using the Cauchy-Schwarz inequality, we get for $j=1,2$,
$$
\int_{V_j}(f-m)^2\, dv_g\ge \frac 1{\vert V_j\vert}\left(\int_{V_j} (f-m)\, dv_g\right)^2= {\vert V_j\vert}\left(\fint_{V_j}f\, dv_g -m\right)^2$$
and then (since $x^2+y^2\ge \frac12 (x-y)^2$)
$$
\int_{V_1} (f-m)^2\, dv_g +\int_{V_2}(f-m)^2\, dv_g\ge \frac12\min( {\vert V_1\vert}, {\vert V_2\vert}) \left(\fint_{V_1}f\, dv_g -\fint_{V_2}f \, dv_g\right)^2$$
which ends the proof.
\end{proof}

The proof of Proposition \ref{Conf2} is more subtle than that of Proposition \ref{Conf1}. The behaviour of a smooth function $f$ away from the boundary has to be taken into account in this case. Indeed, the Steklov eigenfunctions corresponding to index $j\leq b$ can be almost constant on connected components of the boundary. They do not spend a lot of energy "laterally". This is expressed by the situation where $\sum_{j=1}^b
a_j(0)^2$ is large in the proof below.

\begin{proof}[Proof of Proposition \ref{Conf2}]
The situation where the boundary is connected is already treated by
Proposition \ref{Conf1}. Therefore we assume that $b\ge 2$ in what
follows. 
As in the proof of Proposition \ref{Conf1}, we define for every
positive $\varepsilon < \min\{\frac L4, \frac 2L\}$, a Riemannian
metric $g_\varepsilon=h^2_\varepsilon g$ on $M$ where
$h_\varepsilon\geq 1$ is a smooth function which is  identically equal to
  $\varepsilon^{-2}$ in $M\setminus (\Sigma\times [0,L/2])$ and identically equal to 1 on 
  $\Sigma\times [0,\varepsilon).$
 Note that unlike  the conformal deformation used in the proof of
 Proposition \ref{Conf1}, here, the metric $g_\varepsilon$ tends to
 infinity everywhere in the interior of $M$ as $\varepsilon\to 0$.  
 Let $f$ be a smooth function on $M$ with $\int_\Sigma f dv_{g_0}=0$ and
 such that $\int_\Sigma f^2 dv_{g_0}=1$. Our goal is to get a lower
 bound for the Dirichlet energy $R_\varepsilon(f)=\int_M \vert df\vert_{g_\varepsilon}^2
 dv_{g_\varepsilon}$ of the form $C/\varepsilon$. 

\medskip
Let $\Sigma_1, \dots, \Sigma_b$ be the connected
components of the boundary $\Sigma$.
As in the proof of  Proposition \ref{Conf1}, we consider   an
orthonormal basis $\{\phi_j\}_{j\in \N}$ of eigenfunctions of the
Laplacian on $\Sigma$, with $\Delta\phi_j=\lambda_j\phi_j$,
$\lambda_1=\cdots=\lambda_b=0$ and, for every $j\le b$,
$\phi_j=\vert\Sigma_j\vert^{-\frac 12}$ on $\Sigma_j$ and $\phi_j=0$
elsewhere.  The restriction of $f$ to $\Sigma\times [0,L)$  is 
\begin{equation*} 
  f(x,t)=\sum_{j\ge 1} a_j(t)\phi_j(x).
\end{equation*} 
In what follows we treat separately the case where $\sum_{j=1}^b
a_j(0)^2$ is small and the case where it is large. Indeed, following
step by step the proof of Proposition \ref{Conf1}, we get
$$R_\varepsilon(f)\ge \frac {A}\varepsilon \left(1-\sum_{j\le b}a_j(0)^2\right) $$
 with $A=\frac 1{4}\min\{ \lambda_{b+1}, \frac1{4}\}>0$.
 Hence, if $\sum_{j\le b}a_j(0)^2\le \frac 12$, then  
\begin{equation}
R_\varepsilon(f)\ge \frac A{2\varepsilon} .
 \end{equation}

\smallskip
Assume now that $\sum_{j\le b}a_j(0)^2\ge \frac12$ and recall that the number of boundary components $b\ge 2$. We will prove that two of the boundary components, w.l.o.g. $\Sigma_1$ and $\Sigma_2$, are such that $a_1(0)>0$ is large enough while $a_2(0)<0$ is small. Proceeding as in the proof of Proposition \ref{Conf1}, we will first treat the case where both $|a_1(t)|$ and $|a_2(t)|$ decreases quickly away from the boundary. Otherwise, we will appeal to Lemma \ref{lemNeu}. More precisely, let $\Omega$ be the complement in $M$ of $\Sigma\times [0,\frac L2)$. We will prove that 
\begin{equation}\label{conf5"} 
R_\varepsilon(f)\ge \frac B{2\varepsilon} 
 \end{equation} 
 with 
 $$B=\frac{\min\{{\mu(\Omega,g)bL} \  ,  \ \frac{ 1}{2b} \}}{32(b-1)^2}\ \frac{ \min_{j\le b}\vert\Sigma_j\vert{^2}}{ \max_{j\le b}\vert\Sigma_j\vert{^2}},$$
 where, $\mu(\Omega,g)$ is the first positive Neumann eigenvalue of the Laplacian in $\Omega$.
 
Indeed, from $\sum_{j\le b}a_j(0)^2\ge \frac12$ we deduce the existence of  $j_0\le b$ with $a_{j_0}(0)^2\ge \frac1{2b}$. We assume w.l.o.g. that  $j_0=1$ and that $ a_{1}(0)\ge \frac1{\sqrt{2b}}$. From \eqref{conf2} one has $\sum_{2\le j\le b}a_j(0)\vert\Sigma_j\vert^{\frac12}=-a_1(0)\vert\Sigma_1\vert^{\frac12}\le - \frac {\vert\Sigma_1\vert^{\frac12}}{\sqrt{2b}}$. This implies  w.l.o.g. that $a_{2}(0)\vert\Sigma_2\vert^{\frac12}\le -\frac1{b-1}a_1(0)\vert\Sigma_1\vert^{\frac12}\le - \frac {\vert\Sigma_1\vert^{\frac12}}{(b-1)\sqrt{2b}}$.

Now, as in \eqref{conf5} and  \eqref{conf5'}, if there exists $t_0\in[\frac L2, L)\subset [2\varepsilon, L)$ with $ a_1(t_0)\le \frac 12   a_1(0)$, we would have (with $ a_1(0)-a_1(t_0)\ge \frac 12  a_1(0)$)
\begin{align}\label{conf6} 
R_\varepsilon(f)&\ge \int_0^{t_0} a'_1(t)^2h^{n-1}_\varepsilon(t)dt\nonumber\\
&=\int_0^{2\varepsilon} a'_1(t)^2 dt + \int_{2\varepsilon}^{t_0} a'_1(t)^2 \varepsilon^{-2(n-1)}dt\nonumber\\
&\ge\frac 1{16\varepsilon}a_1(0)^2\ge  \frac 1{32 b\varepsilon}.
\end{align}
Similarly, if there exists $t_0\in [\frac L2, L)\subset [2\varepsilon, L)$ with $ a_2(t_0)\ge \frac 12  a_2(0)$, we would have (with 
$ a_2(t_0)-a_2(0)\ge -\frac 12  a_2(0)>0$)
\begin{equation}\label{conf6'} 
R_\varepsilon(f)\ge \int_0^{t_0} a'_2(t)^2h^{n-1}_\varepsilon(t)dt\ge  \frac 1{16\varepsilon}a_2(0)^2\ge  \frac 1{16\varepsilon}\frac {1}{2b(b-1)^2}\frac {\vert\Sigma_1\vert}{\vert\Sigma_2\vert}.
\end{equation}

Let us assume now  that for each $t\in [\frac L2, L)$, $a_1(t)\ge \frac 12  a_1(0)$ and $a_2(t)\le \frac 12  a_2(0)$. We then have, taking into account that, for each $j\le b$, $\int_{\Sigma_j}\phi_i dv_{g_0}=0 $ if $i\ne j$ and   $\int_{\Sigma_j}\phi_j dv_{g_0}=\vert\Sigma_j\vert^{\frac 12}$,
\begin{align*}
 \fint_{\Sigma_1\times[\frac L2, L)} f(x,t)\,dv_g&=\frac{\vert\Sigma_1\vert^{\frac 12}}{|\Sigma_1\times[\frac L2, L)|_g} \int_{\frac L2}^L a_1(t)\,dt
 \ge  \frac 1{2\vert\Sigma_1\vert^{\frac12} } a_1(0) >0 
 \end{align*}
 and 
 \begin{align*}
 \fint_{\Sigma_2\times[\frac L2, L)} f(x,t)\, dv_g=  \frac{\vert\Sigma_2\vert^{\frac 12}}{|\Sigma_2\times[\frac L2, L)|_g} \int_{\frac L2}^L a_2(t)\, dt
 \le  \frac 1{2\vert\Sigma_2\vert^{\frac12} } a_2(0) <0. 
  \end{align*}
 We apply Lemma \ref{lemNeu} to the function $f$ in the complement $\Omega$ in $M$ of $\Sigma\times [0,\frac L2)$.  This leads to
 $$\int_{\Omega} \vert df\vert^2 dv_g \ge \frac{L\mu(\Omega,g)}{4}\min_{j\le b}\vert\Sigma_j\vert\left( \frac 1{2\vert\Sigma_1\vert^{\frac12} } a_1(0) - \frac 1{2\vert\Sigma_2\vert^{\frac12} } a_2(0)\right)^2$$
where $\mu(\Omega,g)$ is the first positive Neumann eigenvalue of the Laplacian in $\Omega$. One has
$$ \frac 1{\vert\Sigma_1\vert^{\frac12} } a_1(0) - \frac 1{\vert\Sigma_2\vert^{\frac12} } a_2(0)\ge  \frac 1{\vert\Sigma_1\vert^{\frac12} } \frac1{\sqrt{2b}} +\frac 1{\vert\Sigma_2\vert } \frac{\vert\Sigma_1\vert^{\frac12} }{(b-1)\sqrt{2b}} \qquad \qquad \qquad \qquad \qquad \qquad $$
$$\qquad \qquad \qquad \qquad \qquad =  \frac {\vert\Sigma_1\vert^{\frac12} }{\sqrt{2b}} \left(\frac 1{\vert\Sigma_1\vert } +\frac 1{(b-1)\vert\Sigma_2\vert }\right) \ge   \frac {b}{(b-1)\sqrt{2b}} \frac{ \min_{j\le b}\vert\Sigma_j\vert^{\frac12}}{ \max_{j\le b}\vert\Sigma_j\vert}.$$
Thus,
$$\int_{\Omega} \vert df\vert^2 dv_g \ge \frac{bL}{32(b-1)^2}{\mu(\Omega,g)}  \frac{ \min_{j\le b}\vert\Sigma_j\vert^{2}}{ \max_{j\le b}\vert\Sigma_j\vert^2}.$$
Since $h_\varepsilon =\varepsilon^{-2}$ on $\Omega$, we get
\begin{align}\label{conf6"} 
R(f_\varepsilon)&=\int_M \vert df\vert^2 h_\varepsilon^{n-1}\, dv_g\nonumber\\
&\ge  \varepsilon^{-2(n-1)}\int_{\Omega} \vert df\vert^2\, dv_g
\ge\frac{\mu(\Omega,g)bL}{32(b-1)^2}  \frac{ \min_{j\le b}\vert\Sigma_j\vert^{2}}{ \max_{j\le b}\vert\Sigma_j\vert^2}\times\frac 1\varepsilon
\end{align}

Combining \eqref{conf6} , \eqref{conf6'}  and \eqref{conf6"}  leads to \eqref{conf5"}.

\bigskip
In conclusion,  the inequality of the proposition holds with  $C={\frac{1}{2}}\min\{A,B\}$.
\end{proof}

\medskip

We are know ready to prove Theorem \ref{thm:main}.
\begin{proof}[Proof of Theorem \ref{thm:main}]
Let $(M,g)$ be a compact Riemannian manifold with $b>0$ boundary components. Given $\delta>0$, define 
$$N(\delta)=\{p\in M\,:\,d_g(p,\Sigma)<\delta\}.$$
The normal exponential map along the boundary defines Fermi coordinates on a neighbourhood $V$ of the boundary. The distance $t$ to the boundary is one of the coordinates. It follows from the Gauss Lemma that the Riemannian metric is expressed by
$g=h_t+dt^2$, where $h_t$ is the restriction of $g$ to the parallel hypersurface $\Sigma_t$ at distance $t$ from the boundary $\Sigma$. In particular it follows from $h_0=g_0$ and continuity that there exists $\delta>0$ such that the restriction of $h_t$ to $\Sigma_t$ is quasi-isometric to $g_0$ with ratio 2 for each $t\in [0,3\delta]$. 
Let $\chi:[0,3\delta]\rightarrow\R$ be a smooth non-decreasing function with value 0 on $[0,\delta]$ and value 1 on $[2\delta,3\delta]$. The metric $g_\delta$ is defined, using Fermi coordinates, by 
$$g_\delta=
\begin{cases}
\chi(t)(h_t+dt^2)+(1-\chi(t))(g_0+dt^2)&\mbox{ in }N(3\delta),\\
g&\mbox{ elsewhere.}
\end{cases}$$
The metric $g_\delta$ is quasi-isometric to $g$ with ratio 2, and it is isometric to the product metric $g_0+dt^2$ on $N(\delta)$.
One can now apply Proposition \ref{Conf1} and Proposition \ref{Conf2} to $(M,g_\delta)$. In both cases, this leads to a family of smooth functions $h_\varepsilon$ such that some eigenvalue $\sigma_j$ satisfies $\lim_{\varepsilon\to 0}\sigma_j(M,h_\varepsilon^2g_\delta)=+\infty$. Note that $h_\varepsilon^2g$ is quasi-isometric to $h_\varepsilon^2g_\delta$ with ratio 2. Therefore, one can apply Proposition \ref{proposition:quasiisocontrol} to deduce that $\lim_{\varepsilon\to 0}\sigma_j(M,h_\varepsilon^2g)=+\infty$.
\end{proof}

\smallskip
The conformal perturbation $g_\varepsilon=h_\varepsilon^2g$ that is used in the proof of Proposition \ref{Conf1} is so that  the diameter of $(M,g_\varepsilon) $ tends to infinity as $\varepsilon\to 0$. Moreover, when $b\ge 2$, then the distance between the components of the boundary also goes to infinity as $\varepsilon\to 0$. 
In the following theorem, a family of metrics on a cylinder is constructed, which coincide with the initial one along the boundary and such that the distance between the two components of the boundary is independent of $\varepsilon$, while $\sigma_2$ becomes arbitrarily large. 
\begin{thm}\label{thm:warpedprod}
  Let $(\Sigma,g_0)$ be a closed connected Riemannian manifold of dimension ${n\geq 3}$ and
  consider the cylinder $M=\Sigma \times [-L,L]$, with $L>0$,  endowed
  with the product metric $g=g_0+dt^2$. For every  $\varepsilon>0$,
  sufficiently small, there exists a Riemannian metric $g_\varepsilon=
  h_\varepsilon^2(t)g_0+dt^2$    which coincides with $g$ in a
  neighborhood of  {the boundary $\Sigma\times\{-L,L\}$} in $M$  and such
  that 
  $$\sigma_2(g_\varepsilon) \ge \frac C \varepsilon,$$ 
  where $C= \frac 14\min( \lambda_2(\Sigma), \frac 1{6})$. The distance between the two boundary components is independent of $\varepsilon$ and may be chosen arbitrarily small.
\end{thm}
{
The condition $n\geq 3$ is necessary. See Remark \ref{rem:necessity} below.
}

\begin{proof}[{Proof of Theorem \ref{thm:warpedprod}}] Given $0<\varepsilon <\frac 14\min(L,\frac 1L)$, let $h_\varepsilon: [-L,L]\to[1,+\infty)$ be a smooth even function such that
\begin{equation}\nonumber
h_\varepsilon(t) =\left\{
\begin{array}{lll}
1 \qquad\;  \text{ in }  [-L, -L+\varepsilon]\cup [L-\varepsilon, L],\\
\varepsilon^{-2} \qquad\;  \text{ in } [ -L+2\varepsilon, L-2\varepsilon].
\end{array}
\right.
\end{equation}
\medskip
On the manifold $M=\Sigma \times [-L,L]$, {define the metric $g_\varepsilon $ by}
$$
g_\varepsilon(t,p)= h^2_\varepsilon(t) g_{0} +dt^2.
$$Let $f$ be a smooth function on $M=\Sigma \times [-L,L]$ with $\int_{\partial M} f\, dv_{g_0}=0$ and such that $\int_{\partial M}f^2\, dv_{g_0}=1$. As before, we develop the function $f$ in Fourier series 
\begin{equation}\label{conf0} 
f(x,t)=\sum_{j\ge 1} a_j(t)\phi_j(x)
\end{equation} 
where $\{\phi_j\}_{j\in \N}$ is an orthonormal basis of eigenfunctions of the Laplacian on $\Sigma$, with $\Delta\phi_j=\lambda_j\phi_j$.
The eigenfunction corresponding to $\lambda_1=0$ is the constant function $\phi_1$ on $\Sigma$ and $\int_\Sigma \phi_j\, dv_{g_0}=0$ for all $j\ge 2$.
The conditions $\int_{\partial M} f\, dv_{g_0}=\int_{\Sigma} (f(x,L)+f(x,-L))\, dv_{g_0}=0$ and $\int_{\partial M}f^2\, dv_{g_0}=\int_{\Sigma} (f(x,L)^2+f(x,-L)^2)\, dv_{g_0}=1$ amount to 
$$a_1(-L)+a_1(L)=0 \quad ,\quad  \sum_{j=1}^{\infty}(a^2_j(-L)+a^2_j(L))=1.$$
From
$
df(x,t)= \sum_{j=1}^{\infty} (a_j'(t)\phi_i(x)\, dt+a_j(t)\,d_0\phi_i(x))
$
we get
$$
R_\varepsilon(f)=\int_{ M} \vert df\vert_{g_\varepsilon}^2\, dv_{g_\varepsilon}$$
where
$$|df|_{g_\varepsilon}^2=\sum_ja_j'(t)^2\phi_j(x)^2+h_\varepsilon^{-2}(t)a_j(t)^2|d_0\phi_j(x)|_{g_0^2},$$
and $dv_{g_\varepsilon}=h_\varepsilon^n\,dv_g$ so that
$$
R_\varepsilon(f)= \sum_{j=1}^{\infty}\int_{-L}^{L}(a_j'^2(t)h^n_\varepsilon(t)+\lambda_j a_j^2(t) h^{n-2}_\varepsilon(t))\,dt.
$$

We set $R^j_\varepsilon(f)=\int_{-L}^{L}(a_j'^2(t)h^n_\varepsilon(t)+\lambda_j a_j^2(t) h^{n-2}_\varepsilon(t))\,dt$ so that $R_\varepsilon(f)= \sum_{j=1}^{\infty}R^j_\varepsilon(f)$.

\smallskip
{\bf Step 1} : For all $j\ge 2$, we will prove that
$$
R^j_\varepsilon(f)\ge  \frac A{\varepsilon}  (a^2_j(-L)+ a^2_j(L)) 
$$
with $A= \min(\frac {\lambda_2}{4}, \frac 1{12})$.

Indeed, let us fix an integer $j$. If there exists  $t_0 \in [-L+2\varepsilon,-L+{3}\varepsilon]$ with  $\vert  a_j(t_0)-a_j(-L) \vert\ge \frac12\vert a_j(-L)\vert $, then,  using the Cauchy-Schwarz inequality
\begin{align}\label{cyl5} 
\int_{-L}^{0}a_j'^2(t) h^n_\varepsilon(t)\,dt &\ge \int_{-L}^{t_0} a_j'^2(t)\,dt\ge \frac 1{t_0+L}\left(\int_{-L}^{t_0} a_j'(t)\,dt\right)^2\nonumber\\
&\ge 
\frac1{3\varepsilon}\left( a_j(t_0)-a_j(-L)\right)^2\ge\frac{1}{12\varepsilon}a^2_j(-L).
\end{align}
Otherwise, $\vert   a_j(t)-a_j(-L)\vert\le   \frac12\vert a_j(-L)\vert $ for all $t\in  [-L+2\varepsilon,-L+{3}\varepsilon]$, which implies  $\vert a_j(t)\vert \ge \frac 12\vert a_j(-L)\vert$  and then
\begin{equation}\label{cyl4} 
\int_{-L}^{0}\lambda_j a_j^2(t) h^{n-2}_\varepsilon(t)\,dt\ge \int_{-L+2\varepsilon}^{-L+3\varepsilon}  \lambda_j a^2_j(t)\varepsilon^{{-2(n-2)}}\,dt\ge \frac {\lambda_j}{4\varepsilon^{{2n-5}}}a^2_j(-L).
\end{equation} 
Thus, in all cases, we have for  $j\ge 2$ (with $\varepsilon^{{2n-5}}\le\varepsilon$ and $\lambda_j\ge\lambda_2$)
\begin{equation}\label{cyl6} 
\int_{-L}^{0}(a_j'^2(t)h^n_\varepsilon(t)+\lambda_j a_j^2(t) h^{n-2}_\varepsilon(t)) dt\ge  \frac A{\varepsilon}a^2_j(-L)
\end{equation}
with $A= \min(\frac {\lambda_2}{4}, \frac 1{12})$.
The same arguments lead to
\begin{equation}\label{cyl6'} 
\int_{0}^{L}(a_j'^2(t)h^n_\varepsilon(t)+\lambda_j a_j^2(t) h^{n-2}_\varepsilon(t)) dt\ge  \frac A{\varepsilon}a^2_j(L).
\end{equation}
Combining \eqref{cyl6} and \eqref{cyl6'}, we get for all $j\ge 2$,
\begin{equation}\label{cyl7} 
R^j_\varepsilon(f)\ge  \frac A{\varepsilon}  (a^2_j(-L)+ a^2_j(L)). 
\end{equation}

{\bf Step 2} For $j=1$, we will also show that
$$R^1_\varepsilon(f)\ge \frac 1{24 \varepsilon}(a^2_1(-L)+ a^2_1(L)).$$
Indeed, recall that we have $a_1(-L)=-a_1(L)$. Using \eqref{cyl5}, we see that if there exists  $t_0 \in [-L+2\varepsilon,-L+{3}\varepsilon]$ with  $\vert  a_1(t_0)-a_1(-L) \vert\ge \frac12\vert a_1(-L)\vert $, then  
$$R^1_\varepsilon(f)\ge\int_{-L}^{0}a_1'^2(t) h^n_\varepsilon(t)\,dt \ge  
\frac{1}{12\varepsilon}a^2_1(-L)= \frac{1}{24\varepsilon}(a^2_1(-L)+ a^2_1(L)).$$
Similarly, if there exists $t_1 \in [L-3\varepsilon,L-2\varepsilon]$ with  $\vert  a_1(t_1)-a_1(L) \vert\ge \frac12\vert a_1(L)\vert $), then  
$$R^1_\varepsilon(f)\ge \int_{0}^{L}a_1'^2(t) h^n_\varepsilon(t)\,dt \ge  
\frac{1}{12\varepsilon}a^2_1(L)= \frac{1}{24\varepsilon}(a^2_1(-L)+ a^2_1(L)).$$
 
Now, assume that we have both
$\vert a_1(t)-a_1(-L) \vert \le \frac{1}{2}\vert a_1(-L)\vert$  for all  $t \in [-L+2\varepsilon,-L+{3}\varepsilon]$  and 
$\vert a_1(L)-a_1(t) \vert \le \frac{1}{2}\vert a_1(L)\vert$ for all $t \in [L-3\varepsilon,L-2\varepsilon]$. Using the Cauchy-Schwarz inequality we get
$$R^1_\varepsilon(f)\ge \int_{-L+{3}{\varepsilon}}^{L- {3}{\varepsilon}} a_1'^2(t)h^n_\varepsilon(t)dt =  
 \varepsilon^{-2n}  \int_{-L+{3}{\varepsilon}}^{L- {3}{\varepsilon}} a_1'^2(t) dt  \ge \frac{\varepsilon^{-2n} }{2L-6\varepsilon} \left(\int_{-L+{3}{\varepsilon}}^{L- {3}{\varepsilon}} a_1'(t) dt \right)^2
$$
$$
 \ge \frac{\varepsilon^{-2n} }{2L} \left(a_1(L- {3}{\varepsilon} ) - a_1(-L+{3}{\varepsilon})   \right)^2
$$
Assume w.l.o.g. $a_1(L)>0$. Then, under the conditions above, we have 
$$a_1(L- {3}{\varepsilon} ) \ge a_1(L) -\vert a_1(L)-a_1(L- {3}{\varepsilon} ) \vert \ge \frac{1}{2} a_1(L)$$ 
and
$$a_1(-L+ {3}{\varepsilon} ) \le a_1(-L) +\vert a_1(-L+ {3}{\varepsilon} ) -a_1(-L) \vert \le a_1(-L) +\frac{1}{2} \vert a_1(-L)\vert=- \frac{1}{2} a_1(L).$$ 
Therefore, $a_1(L- {3}{\varepsilon} ) - a_1(-L+{3}{\varepsilon})\ge a_1(L)$ and, then
$$R^1_\varepsilon(f)\ge \frac{\varepsilon^{-2n} }{2L}a^2_1(L)= \frac{\varepsilon^{-2n} }{4L}(a^2_1(-L)+ a^2_1(L))\ge \frac1\varepsilon (a^2_1(-L)+ a^2_1(L))$$

 \smallskip
In conclusion, we have 
$$R_\varepsilon(f)=\sum_{j\ge 1}R^j_\varepsilon(f)\ge \frac C\varepsilon \sum_{j\ge 1}(a^2_j(-L)+ a^2_j(L))=\frac C\varepsilon$$
with $C=\min(A,\frac 1{24})= \min(\frac {\lambda_2}{4}, \frac 1{24})$.
\end{proof}

{
	\begin{rem}\label{rem:necessity}
		For $n=1$, it follows from Kokarev's bound \eqref{ineq:kokarev} that $$\sigma_2(\Sigma\times[-L,L],g))\leq \frac{4\pi}{\mbox{length}(\Sigma)}.$$
		For $n=2$, and any Riemannian metric of the form $h^2g_0+dt^2$, with $h\equiv 1$ on $\Sigma\times\{-L,L\}$, the following holds:
		$$\sigma_2\leq 2L\lambda_2.$$
		Indeed, in this case one could use the function $f(x,t)=a_2(t)\phi_2(x)$ as a test function and obtain
		$$R(f)=\int_{-L}^L(a_2'^2(t)h^2(t)+\lambda_2a_2^2(t)\,dt.$$
		Using $a_2\equiv 1$ leads to the claimed inequality.
\end{rem}}

\medskip
\begin{rem}
This example shows that if we are far from being a product'', then, immediately, large eigenvalues could appear. Also,  in our construction, the natural projection from $(M,g_\varepsilon)\to [-L,L]$ is   a Riemannian submersion on $[-L,L]$. This implies that $2L$ is also the distance between the two boundaries in $(M,g_\varepsilon)$, for any $\varepsilon$. It is interesting to see that this distance could be very small, without small eigenvalues.
\end{rem}

\subsection*{Acknowledgments}
{The authors are grateful to the anonymous referee for pointing out a mistake in the original proof of Theorem \ref{thm:warpedprod}, which lead to an improvement of the result.}

\bibliographystyle{plain}
\bibliography{biblioCEG}

\end{document}